\input ./style/arxiv-vmsta.cfg
\documentclass[numbers,compress,v1.0.1]{vmsta}

\volume{3}
\issue{2}
\pubyear{2016}
\firstpage{133}
\lastpage{144}
\doi{10.15559/16-VMSTA56}

\startlocaldefs

\newcommand{\rrvert}{\vert}
\newcommand{\llvert}{\vert}
\allowdisplaybreaks

\newtheorem{theo}{Theorem}
\newtheorem{lmm}[theo]{Lemma}
\newtheorem{tv}[theo]{Proposition}

\theoremstyle{definition}

\newtheorem{remark}[theo]{Remark}

\newcommand{\R}{\mathbb{R}}
\newcommand{\Expect}{\mathsf{E}}
\newcommand{\Prob}{\mathsf{P}}
\newcommand{\Real}{\operatorname{Re}}
\newcommand{\sign}{\operatorname{sign}}

\endlocaldefs

\begin{document}

\begin{frontmatter}
\title{Approximations for a solution to stochastic heat equation with
stable noise}

\author{\inits{L.}\fnm{Larysa}\snm{Pryhara}}\email{pruhara7@gmail.com}
\author{\inits{G.}\fnm{Georgiy}\snm{Shevchenko}\corref{cor1}}\email
{zhora@univ.kiev.ua}
\cortext[cor1]{Corresponding author.}
\address{Mechanics and Mathematics Faculty,
Taras Shevchenko National University of Kyiv,
Volodymyrska 64/11,
01601 Kyiv, Ukraine}

\markboth{L. Pryhara, G. Shevchenko}{Approximations for stochastic heat
equation with stable noise}

\begin{abstract}
We consider a Cauchy problem for stochastic heat equation
driven by a real harmonizable fractional stable process $Z$ with Hurst
parameter $H>1/2$ and stability index $\alpha>1$. It is shown that the
approximations for its solution, which are defined by truncating the
LePage series for $Z$, converge to the solution.
\end{abstract}


\received{25 April 2016}
\revised{15 June 2016}
\accepted{15 June 2016}
\publishedonline{30 June 2016}
\end{frontmatter}
\section{Introduction}

Partial differential equations with randomness are widely used to model
physical, chemical, biological phenomena, financial asset prices,
economical processes, etc. The popularity of such models is due to the
combination of deterministic and stochastic features among their
characteristics.

The majority of existing literature is devoted to the case where the
random noise has some Gaussian or sub-Gaussian distribution. To mention
only few papers, a heat equation with Gaussian noise was considered in
\cite{balan-kim,balan-tudor-heat,nualarte-quer,walsh}, and a wave
equation in \cite{balan-tudor-wave,quer-tindel,walsh}. Equations with
sub-Gaussian measures were intensively studied in \cite
{koz-slyvka,koz-veresh}. The articles \cite{radch1,radch2} consider
equations with general stochastic measures.

The research carried out in the cited articles does not allow one to
consider phenomena where the randomness has a heavy-tailed
distribution. But heavy tails are ubiquitous when modeling extreme
risks, so it is quite important to consider equations with heavy-tailed noise.

The main object of this article is a stochastic heat equation in which
the source of randomness is a real harmonizable fractional stable
process $Z$. The solution is understood in the mild sense, with the
integral defined pathwise as a~fractional integral \cite{zahle}.

We consider approximations for the solution of this equation, which are
obtained by truncating the LePage representation series of $Z$. The
main result of this paper is that such representations converge to the
true solution.

The paper is organized as follows. Section~2 contains basic facts about
stable random variables and related processes. It also establishes an
auxiliary analytical lemma. In Section~3, we formulate and prove the
main result of this article.

\section{Preliminaries}

\subsection{Stable random variables and related processes}

In this paper, we consider only symmetric $\alpha$-stable $(S \alpha
S)$ random variables with $\alpha\in(1,2)$. We further provide basic
information about such variables and related objects; for a more
detailed exposition, we refer the reader to \cite{samor-taqqu}.

A random variable $\xi$ is symmetric $\alpha$-stable (S$\alpha$S) with
scale parameter $\sigma^{\alpha}, \sigma\ge0$, if it has the
characteristic function
\[
\Expect \bigl[e^{i\lambda\xi} \bigr]=e^{-|\sigma\lambda|^{\alpha}}.
\]
Given some linear space of S$\alpha$S random variables, the scale
parameter is a quasi-norm on this space, denoted $\|\cdot\|_\alpha$.

To construct families of stable random variables, in particular, stable
random processes, one frequently uses some stable random measures. We
will be interested in the so-called complex rotationally invariant
S$\alpha$S measure $\mu$ on~$\R$. By definition this is a
complex-valued random measure on $\mathcal B(\R)$ with the following properties:
\begin{enumerate}
\item for any Borel set $A \in\mathcal B(\R)$, the random variable
$\Real\mu(A)$ is S$\alpha$S with scale parameter equal to $\lambda
(A)$, the Lebesgue measure of $A$;
\item for any $A \in\mathcal B(\R)$, the random variable $\mu(A)$ is
rotationally invariant, that~is, for any $\theta\in\R$, the
distribution of $e^{i\theta}\mu(A)$ coincides with that of $\mu(A)$;
\item for any disjoint sets $A_1, \ldots,A_n \in\mathcal B(\R)$, the
random variables $\mu(A_1),\dots$, $\mu(A_n)$ are independent.
\end{enumerate}

For a function $f: \R\rightarrow\mathbb{C} $ with
\[
\|f\|_{L^\alpha(\R)}^{\alpha}=\int_{\R}
\bigl|f(x)\bigr|^{\alpha}dx<\infty,
\]
it is possible to define the stochastic integral
\[
I(f)=\int_{\R}f(x)\mu(dx)
\]
such that $\Real I(f)$ is an S$\alpha$S random variable with scale
parameter $\|f\|_{L^\alpha(\R)}^{\alpha}$. In other words,
\[
\bigl\|\Real I(f) \bigr\|_{\alpha}^{\alpha}=\|f\|_{L_\alpha(\R)}^{\alpha},
\]
that is, the real part of the stochastic integral $I(\cdot)$ maps
$L^\alpha(\R)$ isometrically into some family of S$\alpha$S random variables.

Let now $\mathbf T$ be a parametric set. For a measurable function $f
\colon\mathbf T\times\R\rightarrow\mathbb{C}$ such that $f(t,\cdot)
\in L_{\alpha}(\R)$ for all $t \in\mathbf T$, we may define the random
field $\{Z(t),t \in\mathbf T\}$ by
\begin{equation}
\label{eq2} Z(t) = \Real\int_{\R} f(t,x)\mu(dx).
\end{equation}
This random field $Z(t)$ has the so-called LePage series representation
constructed as follows. Let $\varphi$ be an arbitrary positive
probability density on $\R$,
and let the independent families ${\varGamma_{k},\ k\geq1},\  {\xi_{k},\
k\geq1},\ {g_k,\ k\geq1}$, of random variables satisfy:
\begin{enumerate}
\item${\varGamma_{k},\ k\geq1}$, is a sequence of Poisson arrival times
with unit intensity;
\item${\xi_{k} ,\ k\geq1}$, are independent random variables having
density $\varphi$;
\item${g_k,\ k\geq1}$, are independent complex-valued rotationally
invariant Gaussian\break random variables\footnote{Note that the rotational
invariance implies $\Expect[g_k]=0$.} with $\Expect[|\Real
g_k|^\alpha]=1$.
\end{enumerate}
Then the random field ${Z(t),\ t\in\mathbf T}$, defined by \eqref{eq2}
has the same finite-dimensional distributions as
\begin{equation}
\label{eq3} Z'(t)=C_\alpha\Real\sum
_{k\geq1}\varGamma_k^{-{1}/{\alpha}}\varphi(\xi
_k)^{-{1}/{\alpha}}f(t,\xi_k)g_k,
\end{equation}
where
\[
C_\alpha= \biggl(\frac{\varGamma(2-\alpha)\cos\frac{\pi\alpha}{2}}{1-\alpha
} \biggr)^{1/\alpha};
\]
the series converges almost surely for all $t\in\mathbf T$ (see \cite
[Lemma 1]{kono-maejima} and \cite[Theorem~1.4.2]{samor-taqqu}).

In the rest of our paper, $C$ denotes a generic constant whose value
may change from line to line; $C_{a,b,\dots}$ denotes a constant
depending on $a,b,\dots$.

\subsection{Fractional integration}
We will use the pathwise fractional integration; for more detail, see
\cite{samko,zahle}. Let functions $f,g\colon[a,b]\rightarrow\mathbb
{R}$ be such that, for some $\beta\in(0,1)$, the following fractional
derivatives are defined:
\begin{gather*}
\bigl(D_{a+}^{\beta}f\bigr) (x)=\frac{1}{\varGamma(1-\beta)} \Biggl(
\frac
{f(x)}{(x-a)^\beta}+\beta \int_{a}^x
\frac{f(x)-f(u)}{(x-u)^{1+\beta}}du \Biggr)1_{(a,b)}(x),
\\
\bigl(D_{b-}^{1-\beta}g \bigr) (x)=\frac{1}{\varGamma(\beta)} \Biggl(
\frac
{g(x)}{(b-x)^{1-\beta}}+(1-\beta) \int_{x}^b
\frac{g(x)-g(u)}{(x-u)^{2-\beta}}du \Biggr)1_{(a,b)}(x).
\end{gather*}
Provided that $D_{a+}^{\beta}f\in L^1[a,b]$ and $D_{b-}^{1-\beta
}g_{b-}\in
L^\infty[a,b]$, where $g_{b-}(x) = g(x) - g(b)$,
the fractional
integral is defined as
\begin{equation*}
\int_a^bf(x)dg(x)=\int_a^b
\bigl(D_{a+}^{\beta}f\bigr) (x) \bigl(D_{b-}^{1-\beta}g_{b-}
\bigr) (x)dx.
\end{equation*}
It is worth mentioning that for $f\in C^\nu[a,b]$ and $g\in C^\mu[a,b]$
with $\mu+\nu>1$, the fractional
integral $\int_a^bf(x)dg(x)$ is well defined for any
$\beta\in(1-\nu,\mu)$ and equals the limit of Riemann sums.

\subsection{Estimates of Fourier-type integrals}

The following result specifies the rate of convergence in the
Riemann--Lebesgue lemma. It may be known that, for example,\ for
periodic functions and integer parameter, this is Zygmund's theorem;
however, we failed to find it in the literature. Moreover, a similar
reasoning will be used later in the proof of our main results, so we
found it suitable to present its proof.
\begin{lmm}\label{l1}
Let $f\in C{[a,b]}$ and $h\colon[0,+\infty)\rightarrow[0,+\infty)$ be
a nondecreasing function such that $|f(x)-f(y)|\leq h
(|x-y|)$ for all $x,y \in[a,b]$. Then, for any
nonzero $\lambda\in\R$,
\[
\Biggl|\int_a^bf(x)e^{i\lambda x}dx \Biggr| \leq 3(b-a)h
\bigl(|\lambda|^{-1} \bigr)+ 2|\lambda|^{-1} \sup
_{x\in[a,b]} \bigl|f(x) \bigr|.
\]
\end{lmm}
\begin{proof}
If $|\lambda|\leq(b-a)^{-1}$, then
\begin{align*}
\Biggl|\int_a^bf(x)e^{i\lambda x}dx \Biggr| &{}\leq \int
_a^b \bigl|f(x)e^{i\lambda x} \bigr|dx
\\
&{}
\leq \sup_{x\in[a,b]} \bigl|f(x)\bigr|(b-a)
\leq |\lambda|^{-1}\sup_{x\in[a,b]} \bigl|f(x)\bigr|.
\end{align*}
Otherwise, set $n=[|\lambda|(b-a)]+1$, so that
$|\lambda|(b-a)\leq n\leq2|\lambda|(b-a)$.
Consider the equipartition of $[a,b]$ by points
$x_k=a+{(b-a)k}/{n}$,
$k=0,\dots,n$. Then, for any
$|\lambda|\geq(b-a)^{-1}$, the following relations hold:
\begin{align*}
&{}
\Biggl|\int_a^bf(x)e^{i\lambda x}dx \Biggr|
\\
&\quad{}\leq
\Biggl|\sum_{k=1}^{n}\int_{x_{k-1}}^{x_k}f(x_k)e^{i\lambda x}dx \Biggr|+
\Biggl|\sum_{k=1}^{n}\int_{x_{k-1}}^{x_k}\bigl(f(x)-f(x_k) \bigr)e^{i\lambda x}dx \Biggr|
\\
&\quad{}\leq
\Biggl|\sum_{k=1}^{n}f(x_k)\frac{e^{i\lambda x_k}-e^{i\lambda x_{k-1}}}{i\lambda} \Biggr|
+\sum_{k=1}^{n}\int_{x_{k-1}}^{x_k} \bigl|f(x)-f(x_k) \bigr|dx
\\
&\quad{}\leq
\Biggl|\sum_{k=1}^{n}f(x_k)\frac{e^{i\lambda x_k}}{\lambda}-\sum_{k=0}^{n-1}f(x_{k+1})\frac{e^{i\lambda x_k}}{\lambda} \Biggr|
+h \biggl(\frac{b-a}{n} \biggr) (b-a)
\\
&\quad{}\le
\biggl|f(b)\frac{e^{i\lambda b}}{\lambda}-f(a)\frac{e^{i\lambda a}}{\lambda} \biggr|+
\Biggl|\sum_{k=1}^{n-1}\bigl(f(x_k)-f(x_{k+1})\bigr)\frac{e^{i\lambda x_k}}{\lambda} \Biggr|
+h \biggl(\frac{b-a}{n} \biggr) (b-a)&
\\
&\quad{}\leq
\frac{2\sup_{x\in[a,b]}|f(x)|}{|\lambda|}+
\sum_{k=1}^{n-1}\frac{1}{|\lambda|}h \biggl(\frac{b-a}{n} \biggr)
+h \biggl(\frac{b-a}{n}\biggr) (b-a)
\\
&\quad{}\leq
\frac{2\sup_{x\in[a,b]}|f(x)|}{|\lambda|}+
\sum_{k=1}^{n-1}\frac{2(b-a)}{n}h \biggl(\frac{1}{|\lambda|} \biggr)
+h \biggl(\frac{1}{|\lambda|}\biggr) (b-a)
\\
&\quad{}\leq
\frac{2\sup_{x\in[a,b]}|f(x)|}{|\lambda|}+3(b-a)h \biggl(\frac{1}{|\lambda|} \biggr).\qedhere
\end{align*}
\end{proof}

\section{Stochastic heat equation with stable noise and its approximations}

Consider a Cauchy problem for the one-dimensional heat equation
\begin{equation}
\label{eq1} %
\begin{cases}
d_tU(t,x)=\dfrac{1}{2}\dfrac{\partial^2}{\partial x^2}U(t,x)dt+\sigma
(t,x)dZ(t), \quad t>0, x\in\R,\\
U(0,x)=U_0(x).
\end{cases} %
\end{equation}
Here $\sigma(t,x)$ is a bounded function that is jointly H\"older
continuous of order $\gamma\in(1/2,1)$, that is,
\begin{equation}
\label{eqG}
\bigl|\sigma(t_1,x_1)-\sigma(t_2,x_2)\bigr|
\leq C
\bigl(|t_1-t_2|^\gamma+|x_1-x_2|^\gamma\bigr),
\end{equation}
and $U_0$ is a bounded measurable function. The random force in this
equation is a real harmonizable fractional stable process
\[
Z(t)=\Real\int_{\R}\frac{e^{itx}-1}{\llvert x\rrvert ^{1/\alpha+H}} M(dx),
\]
where $M$ is a complex rotationally invariant S$\alpha$S measure $\mu$
on $\R$, defined in Section~2.1, and $H\in(1/2,1)$ is the Hurst
parameter of the process. In what follows, we denote
\[
f(t,x)=\frac{e^{itx}-1}{\llvert x\rrvert ^{1/\alpha+H}}.
\]
It is well known (see, e.g., \cite{samor-taqqu}) that $Z$ is an
$H$-self-similar process having a continuous modification; henceforth,
we assume that $Z$ itself is continuous. Moreover, it is almost surely
pathwise H\"older continuous with any exponent $\gamma\in(0, H)$ (see
\cite{kono-maejima}).

We consider Eq.~\eqref{eq1} in the mild sense. We recall that a mild
solution is given by the variation-of-constants formula
\begin{equation}
\label{eqM} U(t,x)=\int_{\R}\rho(t,x-y)U_0(y)dy+
\int_0^tdZ(s)\int_{\R}\rho
(t-s,x-y)\sigma(s,y)dy,
\end{equation}
where $\rho(t,x)=(4\pi t)^{-{1}/{2}}\exp\{-\frac{|x|^2}{4t}\}$.

\begin{theo}\label{thm:exuniq}
The Cauchy problem \eqref{eq1} has a solution given by \eqref{eqM},
where the integral with respect to $Z$ is understood as a fractional
integral.\vadjust{\eject}
\end{theo}
\begin{proof}
Take some $\beta\in(1-H,\gamma)$. Fix $t>0$ and $x\in\mathbb{R}^d$
and denote $f(s) = \int_{\R}\rho(t-s,x-y)\sigma(s,y)dy$, $s\in[0,t]$.
In view of the H\"older continuity of $Z$, the fractional derivative
$D^{1-\beta}_{t-}Z_{t-}$ is almost surely bounded on $[0,t]$. So in
order to prove the claim, we only need to show that $D^\beta_{0+}f$ is
integrable on $[0,t]$. To this end, write
\begin{gather*}
\int_0^t \bigl|D^\beta_{0+}f(s)
\bigr|ds \le C\int_0^t \Biggl(\frac{|f(s)|}{s^\beta} +
\int_0^s \frac{|f(s) - f(u)|}{(s-u)^{1+\beta}}du \Biggr)ds.
\end{gather*}
Since $\sigma$ is bounded, so is $f$, supplying the finiteness of the
first integral. To establish that of the second one, use the change of
variable $y = x + z\sqrt{t-s}$ and note that $\rho(t,x) = \rho(1,x/\sqrt
{t})/\sqrt{t}$ to represent $f$ as
\[
f(s) = \int_{\R}\rho(1,z)\sigma(s,x + z\sqrt{t-s})dz.
\]
Therefore, for any $u<s<t$,
\begin{align*}
\bigl|f(s) - f(u) \bigr| &{}\leq \int_{\R}\rho(1,z)
\bigl|\sigma(s,x+z
\sqrt{t-s})-\sigma(s,x+z\sqrt{t-u}) \bigr|dz
\\
&\quad{}+ \int_{\R}\rho(1,z) \bigl|\sigma(s,x+z\sqrt{t-u})-
\sigma(u,x+z\sqrt{t-u}) \bigr|dz
\\
&{} =:I_1+I_2.
\end{align*}
Thanks to \eqref{eqG}, $I_2\le C(s-u)^\gamma$ and
\begin{align*}
I_1 &{}\leq \int_{\R}\rho(1,z)|z\sqrt{t-s}-z
\sqrt{t-u}|^\gamma dz \leq C |\sqrt{t-s}-\sqrt{t-u}|^\gamma
\\
&{}= C \biggl(\frac{s-u}{\sqrt{t-s}+\sqrt{t-u}} \biggr)^\gamma \leq C
(t-u)^{-\gamma/2}(s-u)^\gamma.
\end{align*}
Consequently,
\begin{align*}
\int_0^t \int_0^s
\frac{|f(s) - f(u)|}{(s-u)^{1+\beta}}du\,ds &{}\le C \int_0^t \int
_0^s (t-u)^{-\gamma/2}(s-u)^{\gamma-1-\beta}du
\,ds
\\
&{}\le C \int_0^t (t-s)^{-\gamma/2}
s^{\gamma-\beta}ds<\infty,
\end{align*}
which concludes the proof.
\end{proof}

The process $Z$ is a particular example of a random field given by
\eqref{eq2}. In view of this, we assume that $Z$ is given by its LePage
series representation \eqref{eq3} corresponding to the density
\[
\varphi(x)=K_{\eta}|x|^{-1} \bigl|\log|x|+1 \bigr|^{-1-\eta},
\]
where $\eta$ is some positive number, and $K_{\eta}=(\int_{\R}
|x|^{-1}|\ln|x|+1|^{-1-\eta}dx)^{-1}$ is a normalizing constant.

To simplify the following reasoning, we assume that
\[
(\varOmega, \mathcal F, \Prob)= (\varOmega_{\varGamma}\otimes \varOmega_{\xi}
\otimes\varOmega_{g}, \mathcal F_{\varGamma} \otimes\mathcal
F_{\xi} \otimes\mathcal F_{g}, \Prob_{\varGamma}\otimes
\Prob_{\xi} \otimes\Prob_{g})
\]
and
\[
\varGamma_{k}(\omega)=\varGamma_{k}(\omega_{\varGamma}),\
\xi_{k}(\omega)=\xi _{k}(\omega_{\xi}),\
g_{k}(\omega)=g_{k}(\omega_{g})
\]
for all $\omega= (\omega_{\varGamma},\omega_{\xi}, \omega_{g})\in\varOmega,\ k\geq1$.

Let us consider the approximation of the process $Z$ by partial sums of
its LePage series \eqref{eq3}. Specifically, define
\[
Z_N(t)=C_\alpha\Real\sum_{k=1}^N
\varGamma_k^{-1/\alpha}\varphi(\xi _k)^{-1/\alpha}f(t,
\xi_k)g_k.
\]
The following result establishes a uniform, in $N$, H\"older continuity
of the family $\{Z^N,N\ge1\}$.
\begin{tv}\label{tv1}
For any $\theta\in(0,H)$ and $T>0$, there is an almost surely finite
random variable $C=C_{\theta,T}(\omega)$ such that, for all $ N\geq1$
and $t,s \in[0,T]$, we have the following inequality:
\[
\bigl|Z_N(t)-Z_N(s) \bigr|\leq C_\theta(\omega) |t-s
|^\theta.
\]
\end{tv}
\begin{proof}
For fixed $\omega_{\varGamma} \in\varOmega_{\varGamma}$, $\omega_{\xi} \in
\varOmega_{\xi}$, and $ t>0$, the sequence
\[
\bigl\{Z_N(t)=Z_N \bigl(t, (\omega_{\varGamma},
\omega_{\xi}, \omega{g}) \bigr), N\geq1 \bigr\}
\]
is a martingale on $(\varOmega_{g}, F_{g}, \Prob_{g})$.
Then, for any $N_{0}\geq1$ and $t,s>0$, the Doob inequality yields
\[
\Expect_{g} \Bigl[ \sup_{1\leq N\leq N_{0}}
\bigl(Z_N(t)-Z_N(s) \bigr)^2 \Bigr] \leq C
\Expect_{g} \bigl[ \bigl(Z_{N_0}(t)-Z_{N_0}(s)
\bigr)^2 \bigr].
\]
Letting $N_0 \rightarrow\infty$ and applying the Fatou lemma, we get
\[
\Expect_{g} \Bigl[\sup_{N\geq1} \bigl(Z_N(t)-Z_N(s)
\bigr)^2 \Bigr]\leq C\Expect_{g} \bigl[ \bigl(Z(t)-Z(s)
\bigr)^2 \bigr].
\]
Using further a reasoning similar to that used in \cite[Theorem
1]{kono-maejima}, we get, for any $t,s\in[0,T]$,
\[
\Expect_{g} \Bigl[\sup_{N\geq1} \bigl(Z_N(t)-Z_N(s)
\bigr)^2 \Bigr]\leq C_T(\omega)\llvert t-s\rrvert
^{2H}\bigl\llvert \log\llvert t-s\rrvert +1\bigr\rrvert ^a
\]
with some $a>0$.
Consequently, for any $\theta\in(0,H)$,
\[
\sup_{N\geq1}\bigl\llvert Z_N(t)-Z_N(s)
\bigr\rrvert \leq C_{\theta,T}(\omega)\llvert t-s\rrvert ^\theta,
\]
as required.
\end{proof}
Taking into account that the processes $Z^N$ approximate $Z$, it is
natural to consider corresponding approximations of a mild solution to
\eqref{eq1}:
\begin{gather*}
U_N(t,x)=\int_{\R}\rho(t,x-y)U_0(y)dy+
\int_0^tdZ_N(s)\int
_{\R}\rho (t-s,x-y)\sigma(s,y)dy.
\end{gather*}
The following theorem is the main result of this paper.
\begin{theo}\label{t1}
For any $t \ge0$ and $x \in\R$, we have the convergence
\[
U_{N}(t,x)\rightarrow U(t,x),\ N\rightarrow\infty,
\]
almost surely.
\end{theo}
\begin{proof}Fix arbitrary numbers $t>0$ and $x\in\R$. Let us define
the functions $v_k(t,x)$ corresponding to the terms of LePage series:
\[
v_k(t,x)=C_\alpha\varGamma_k^{-1/\alpha}
\varphi(\xi_k)^{-1/\alpha}\Real \int_0^t
d_sf(s,\xi_k)\int_{\R}\rho(t-s,x-y)
\sigma(s,y)dy.
\]
Then
\[
U_N(t,x)=\int_{\R}\rho(t,x-y)U_0(y)dy+
\sum_{k=1}^Nv_k(t,x).
\]
As the first step of our proof, we establish the almost sure
convergence of the series $\sum_{k=1}^\infty v_k(t,x)$ for all $t \in
[0,T]$ and $ x \in\R$.

Let us transform the differential
\begin{gather*}
d_sf(s,\xi_k)=d_s \biggl(
\frac{e^{is\xi_k}-1}{\llvert \xi_k\rrvert ^{1/\alpha+H}} \biggr)
=\frac{e^{is\xi_k}\cdot i\xi_k}{\llvert \xi
_k\rrvert ^{1/\alpha+H}}dt= \frac{ie^{is\xi_k} \sign\xi_k}{\llvert \xi
_k\rrvert ^{1/{\alpha}+H-1}}dt.
\end{gather*}
Then we have
\begin{align*}
v_k(t,x)&{}=C_\alpha\varGamma_k^{-1/\alpha} \Real
\Biggl[ i\frac{\varphi(\xi
_k)^{-1/\alpha}\sign\xi_k}{\llvert \xi_k\rrvert ^{1/\alpha+H-1}}g_k
\\
&\quad{}\times
\int_0^t\int_{\R}
\rho(t-s,x-y)\sigma(s,y)dy\,e^{is\xi_k}ds \Biggr].
\end{align*}
Hence,
\begin{align*}
\Expect_{g} \bigl[\bigl\llvert v_k(t,x)\bigr\rrvert
^2 \bigr]&{}\ge C_\alpha^2\varGamma
_k^{-2/\alpha}\frac{\varphi(\xi_k)^{-2/\alpha}}{\llvert \xi_k\rrvert ^{2/\alpha+2H-2}}
\\
&\quad{}\times\Biggl\llvert \int_0^t\int
_{\R} \rho(t-s,x-y)\sigma(s,y)dy\,e^{is\xi
_k}ds\Biggr
\rrvert ^2.
\end{align*}
Let us estimate the last integral:
\begin{align*}
&{}\Biggl\llvert \int_0^t\int_{\R}\rho(t-s,x-y)\sigma(s,y)dy\,e^{is\xi_k}ds\Biggr\rrvert\\
&\quad{}\leq
\Biggl\llvert \int_0^t\int_{\R}\rho(t-s,x-y)\sigma(s,x)dy\,e^{is\xi_k}ds\Biggr\rrvert
\\
&\qquad{}+
\Biggl\llvert \int_0^t\int_{\R}\rho(t-s,x-y) \bigl(\sigma(s,y)-\sigma(s,x) \bigr)dy\,e^{is\xi_k}ds\Biggr\rrvert
=:I_1 + I_2.
\end{align*}
Since $\int_{\R}\rho(t-s,x-y)dy=1$, we have
\[
I_1 = \Biggl\llvert \int_0^t\int
_{\R}\rho(t-s,x-y)\sigma(s,x)dy\,e^{is\xi
_k}ds\Biggr
\rrvert =\Biggl\llvert \int_0^t
\sigma(s,x)e^{is\xi_k}ds\Biggr\rrvert .\vadjust{\eject}
\]
First, assume that $|\xi_k|\ge t^{-1}$. Using Lemma \ref{l1}, the last
expression admits the following estimate:
\[
\Biggl\llvert \int_0^t\sigma(s,x)e^{is\xi_k}ds
\Biggr\rrvert \leq3t C\llvert \xi_k\rrvert ^{-\gamma}+2\llvert
\xi_k\rrvert ^{-1}\sup_{s \in[0,t]}\bigl\llvert
\sigma (s,x) \bigr\rrvert \le C\llvert \xi_k\rrvert ^{-\gamma}.
\]
Let us now estimate $I_2$, taking into account that $\rho(t,x) = \rho
(1,x/\sqrt{t})/\sqrt{t}$:
\begin{align*}
&{}\Biggl\llvert \int_0^t\int_{\R}\rho(t-s,x-y) \bigl(\sigma(s,y)-\sigma(s,x)\bigr)dye^{is\xi_k}ds\Biggr\rrvert
\\[-1.5pt]
&\quad{}=
\Biggl\llvert \int_0^t\int_{\R}\rho\biggl(1,\frac{x-y}{\sqrt{t-s}}\biggr) \bigl(\sigma(s,y)-\sigma(s,x)\bigr)\frac{dy}{\sqrt{t-s}}e^{is\xi_k}ds\Biggr\rrvert
\\[-1.5pt]
&\quad{}=
\Biggl\llvert\int_0^t\int_{\R}\rho(1,z)\bigl(\sigma(s,x+z\sqrt{t-s})-\sigma (s,x)\bigr)dze^{is\xi_k}ds\Biggr\rrvert
\\[-1.5pt]
&\quad{}=\Biggl\llvert \int_0^t\tau(s)e^{is\xi_k}\Biggr\rrvert =:I_3,
\end{align*}
where
\[
\tau(s)=\int_{\R}\rho(1,z) \bigl(\sigma(s,x+z\sqrt{t-s})-
\sigma(s,x)\bigr)dz.
\]
We further estimate
\begin{align*}
\bigl\llvert \tau(s_1)-\tau(s_2)\bigr\rrvert &{}=
\Biggl| \int_{\R}\rho(1,z) \bigl(\sigma (s_1,x+z\sqrt{t-s_1})-\sigma(s_1,x) \bigr)
\\[-2.5pt]
&\quad{}- \bigl(\sigma(s_2,x+z\sqrt{t-s_2})-\sigma(s_2,x) \bigr)dz \Biggr|
\\[-2.5pt]
&{}\leq
\biggl\llvert \int_{\R}\rho(1,z)
\bigl(\sigma(s_1,x+z\sqrt{t-s_1})-\sigma
(s_1,x+z\sqrt{t-s_2}) \bigr)dz\biggr\rrvert
\\[-1.5pt]
&\quad{}+\biggl\llvert \int_{\R}\rho(1,z) \bigl(\sigma(s_1,x+z\sqrt{t-s_2})-\sigma (s_2,x+z\sqrt{t-s_2}) \bigr)dz\biggr\rrvert
\\[-1.5pt]
&\quad{}+\int_{\R}\rho(1,z)\bigl\llvert \sigma(s_1,x)-\sigma(s_2,x)\bigr\rrvert dz
=:J_1+J_2+J_3.
\end{align*}
Thanks to the H\"older continuity \eqref{eqG}, \ $J_3\leq
C(s_1-s_2)^\gamma$ for $s_2<s_1$.
Further, similarly to the proof of Theorem \ref{thm:exuniq},
$J_1\leq C (s_1-s_2)^\gamma(t-s_s)^{-\gamma/2}$ and
$J_2\leq C (s_1-s_2 )^\gamma$.
Consequently, for $s_2<s_1$,
\[
\bigl\llvert \tau(s_1)-\tau(s_2)\bigr\rrvert \leq C
(s_1-s_2 )^\gamma (t-s_2
)^{-\gamma/2}.
\]
Similarly to $J_1$, $\llvert \tau(s)\rrvert \leq C (t-s)^{\gamma/2}$.

Further, recall that $|\xi_k|\ge t^{-1}$. As in the proof of Lemma \ref
{l1}, set $n= [\llvert \xi_k\rrvert  t ]+1$, so that $\llvert \xi
_k\rrvert t\leq n\leq2\llvert \xi_k\rrvert t$, and define the equidistant
partition of $[0,t]$: $t_j={tj}/{n}, j=0,\dots,n$. Then
\begin{align*}
I_3&{}=\Biggl\llvert \sum_{j=1}^{n}\int_{t_{j-1}}^{t_j}\tau(s)e^{is\xi_k}ds\Biggr\rrvert
\\
&{}\leq\Biggl\llvert \sum_{j=1}^{n}\int_{t_{j-1}}^{t_j}\tau(t_{j-1})e^{is\xi_k}ds\Biggr\rrvert +\sum_{j=1}^{n}\int_{t_{j-1}}^{t_j}\bigl\llvert \tau(s)-\tau(t_{j-1})\bigr\rrvert ds
\\
&{}\leq\Biggl\llvert \sum_{j=1}^{n}\tau(t_{j-1}) \biggl(\frac{e^{it_{j}\xi_k}}{\xi_k}-\frac{e^{it_{j-1}\xi_k}}{\xi_k} \biggr)\Biggr\rrvert +C\sum_{j=1}^{n}\int_{t_{j-1}}^{t_j} (t-t_{j-1} )^{-\gamma/2}(s-t_{j-1} )^\gamma ds
\\
&{}\leq\biggl\llvert \frac{\tau(0)}{\xi_k}\biggr\rrvert +\biggl\llvert\frac{\tau(t)}{\xi_{k}}\biggr\rrvert +\Biggl\llvert \sum_{j=1}^{n}\frac{e^{it_{j}\xi_k}}{\xi_k} \bigl(\tau (t_j)-\tau(t_{j-1}) \bigr)\Biggr\rrvert
\\
&\quad{}+ C\sum_{j=1}^{n} (t-t_{j-1})^{-\gamma/2} (t_j-t_{j-1} )^\gamma\frac{t}{n}
\\
&{}\leq C \Biggl(\llvert \xi_k\rrvert ^{-1}+\llvert\xi_k\rrvert ^{-1}\sum_{j=1}^{n}\biggl(\frac{t}{n} \biggr)^\gamma+ \biggl(\frac{t}{n}\biggr)^\gamma \Biggr) \leq C\llvert \xi_k\rrvert^{-1} \bigl(1+n^{1-\gamma} \bigr)
\\
&{}\leq C\llvert \xi_k\rrvert ^{-1} \bigl(1+\llvert
\xi_k\rrvert ^{1-\gamma} \bigr)\le C\llvert \xi_k \rrvert ^{-\gamma}.
\end{align*}
Otherwise, if $\llvert \xi_k\rrvert \leq t^{-1}$, then
\[
I_3\leq\Biggl\llvert \int_0^t
\tau(s)e^{is\xi_k}ds\Biggr\rrvert \leq\int_0^t
\bigl\llvert \tau (s)\bigr\rrvert ds\leq C.
\]
Then, for $\llvert \xi_k\rrvert \geq t^{-1}$,
\begin{align*}
\Expect_{g} \bigl[\bigl\llvert v_k(t,x)\bigr\rrvert
^2 \bigr]
&{}\le C_\alpha^2\varGamma
_k^{-2/\alpha}\frac{\varphi(\xi_k)^{-2/\alpha}}{\llvert \xi_k\rrvert ^{2/\alpha+2H-2}}\Biggl\llvert \int
_0^t\int_{\R}\rho(t-s,x-y)
\sigma (s,y)dy\,e^{is\xi_k}ds\Biggr\rrvert ^2
\\
&{}\leq C\varGamma_k^{-2/\alpha}\llvert \xi_k\rrvert
^{2-2H-2\gamma}\bigl\llvert \ln \llvert \xi_k\rrvert +1\bigr\rrvert
^{2(1+\eta)/\alpha},
\end{align*}
whereas, for $\llvert \xi_k\rrvert < t^{-1}$,
\[
\Expect_{g} \bigl[\bigl\llvert v_k(t,x)\bigr\rrvert
^2 \bigr]\leq C\varGamma _k^{-2/\alpha}\llvert
\xi_k\rrvert ^{2-2H}\bigl\llvert \ln\llvert
\xi_k\rrvert +1\bigr\rrvert ^{2(1+\eta)/\alpha}.
\]
Hence,
\begin{align*}
\Expect_{\xi, g} \bigl[\bigl\llvert v_k(t,x)\bigr\rrvert
^2 \bigr]
&{}\leq C\varGamma _k^{-2/\alpha} \biggl[\int
_{\llvert y\rrvert \geq t^{-1}}\llvert y\rrvert ^{2-2H-2\gamma} \bigl(\bigl\llvert \ln
\llvert y\rrvert \bigr\rrvert +1 \bigr)^{2(1+\eta
)/\alpha}\varphi(y)dy
\\
&\quad{}+\int_{\llvert y\rrvert <t^{-1}}\llvert y\rrvert ^{2-2H} \bigl(\bigl
\llvert \ln\llvert y\rrvert \bigr\rrvert +1 \bigr)^{2(1+\eta)/\alpha}\varphi(y)dy
\biggr]
\\
&{}=C \varGamma_k^{-2/\alpha} \biggl[\int_{\llvert y\rrvert \geq t^{-1}}
\llvert y\rrvert ^{1-2H-2\gamma} \bigl(\bigl\llvert \ln\llvert y\rrvert \bigr
\rrvert +1 \bigr)^{(-1+2/\alpha)(1+\eta)}dy
\\
&\quad{}+\int_{\llvert y\rrvert <t^{-1}}\llvert y\rrvert ^{1-2H} \bigl(\bigl
\llvert \ln\llvert y\rrvert \bigr\rrvert +1 \bigr)^{(-1+2/\alpha
)(1+\eta)}dy \biggr].
\end{align*}
The first integral converges since $1-2H-2\gamma<-1$, whereas the
second one converges since $1-2H>-1$. Therefore,
\[
\sum\limits
_{k=1}^{\infty}\Expect_{\xi, g} \bigl[
\bigl\llvert v_k(t,x)\bigr\rrvert ^2 \bigr]\leq\sum
_{k=1}^{\infty}C_{k}
\varGamma_k^{-2/\alpha}.
\]

By the strong law of large numbers, $\varGamma_k\sim\frac{1}{k},\
k\rightarrow+\infty$, $\Prob_\varGamma$-almost surely. Therefore, $\sum_{k=1}^{\infty}\Expect_{\xi, g}  [\llvert v_k(t,x)\rrvert ^2
]<\infty$ \  $\Prob_\varGamma$-almost surely.

In particular, \rule{0pt}{10pt}$\sum_{k=1}^{\infty}\Expect_{g}  [\llvert v_k(t,x)\rrvert ^2  ]$ converges $\Prob_\xi\otimes\Prob_\varGamma
$-almost surely. For fixed $\omega_\xi\in\varOmega_\xi$ and $ \omega_\varGamma
\in\varOmega_\varGamma$, the random variables $ \{v_k(t,x), k\geq1
\}$ are independent centered Gaussian random variables; moreover,
\[
\sum\limits
_{k=1}^{\infty}\Expect \bigl[\bigl\llvert
v_k(t,x)\bigr\rrvert ^2 \bigr]<\infty.
\]
Then, by the Kolmogorov theorem, $\sum_{k=1}^{\infty}v_k(t,x)$
converges $\Prob_\xi\otimes\Prob_\varGamma\otimes\Prob_g$-almost
surely, as claimed.

It remains to prove that the sum $U_0(t,x) + \sum_{k=1}^{\infty
}v_k(t,x)$ is equal to $U(t,x)$. We first show that $Z_N(t)\rightarrow
Z(t),\ N\rightarrow\infty$, almost surely in $C^\theta[0,T]$ for any
$T>0$, $\theta\in(0,H)$. Taking into account that, for any $t \in
[0,T]$, $Z_N(t)\rightarrow Z(t), N\rightarrow\infty$, almost surely, we
get that there is a set $\varOmega_0 \subset\varOmega$ such that $\mathsf
{P}(\varOmega_0) = 1$ and $Z_N(t)\rightarrow Z(t),\ N\rightarrow\infty$,
for any $t \in[0,T]\cap\mathbb Q$, $\omega\in\varOmega_0$. Thanks to
Proposition \ref{tv1}, the sequence $ \{Z_N, N\geq1 \}$ is
almost surely bounded in $C^\theta[0,T]$ for any $\theta\in(0,H)$
and $T>0$. Therefore, it is precompact in each of these spaces. Take
arbitrary $\theta\in(0,H)$. Without loss of generality, the sequence
$ \{Z_N, N\geq1 \}$ is precompact in $C^\theta[0,T]$ for any
$\omega\in\varOmega_0$. Fix $\omega\in\varOmega_0$ and let $\{
Z_{N_k},k\ge1\}$ be any subsequence of $\{Z_N,N\ge1\}$. In view of
precompactness, it must contain a subsequence convergent in $C^\theta$;
to avoid cumbersome notation, we assume that $Z_{N_k} \to Y$, $k\to
\infty$. In particular, $Z_{N_k}(t)\to Y(t)$, $k\to\infty$, $t \in
[0,T]\cap\mathbb Q$. In view of continuity, $Z(t) = Y(t)$ for any
$t\in[0,T]$. Since any subsequence of $\{Z_N,N\ge1\}$ contains a
subsequence convergent to $Z$ in $C^\theta[0,T]$, the sequence itself
converges to $Z$.

Now, thanks to the integrability of the fractional derivative of $f(s)
= \int_{\R}\rho(t-s,x-y)\sigma(s,y)dy$, which was shown in the proof of
Theorem \ref{thm:exuniq}, the convergence established in the previous
paragraph yields
\[
\int_0^tdZ_N(s)\int
_{\R}\rho(t-s,x-y)\sigma(s,y)dy \rightarrow\int
_0^tdZ(s)\int_{\R}
\rho(t-s,x-y)\sigma(s,y)dy
\]
as $n\rightarrow\infty$, concluding the proof.
\end{proof}

\begin{remark}
It is possible to consider \eqref{eq1} with $Z$ being a real
harmonizable multifractional stable motion considered in \cite
{dozzi-shev}. Making some minor changes, one can show that
Theorem~\ref{t1} is valid in this case as well.
\end{remark}

\section*{Acknowledgements}

The authors thank an anonymous referee for his careful reading of the
manuscript and valuable remarks, which led to a significant improvement of the article.

%

\end{document}